\documentclass[11pt,reqno]{amsart}

\usepackage[utf8]{inputenc}

\usepackage{amsmath,amssymb,amsthm}
\usepackage{mathrsfs}

\usepackage{indentfirst}
\usepackage{setspace}
\usepackage{enumerate}

\usepackage{graphicx}
\usepackage{tikz}
\usepackage{float}

\usepackage{xcolor}
\usepackage{colortbl}
\usepackage{multirow}

\usepackage[
    colorlinks=true,
    linkcolor=purple,
    citecolor=blue,
    urlcolor=magenta
]{hyperref}

\definecolor{sectionlink}{RGB}{0,100,200}

\newtheorem{theo}{Theorem}[section]
\newtheorem{lem}{Lemma}[section]

\newcommand{\ol}{\overline}

\numberwithin{equation}{section}

\begin{document}

\title[Sharp Coefficient Estimates for the Exponential Starlike class $\mathcal{S}_{ex}^{\ast}$]{Sharp Coefficient Estimates for the Exponential Starlike class $\mathcal{S}_{ex}^{\ast}$}
\author[P. Das and N. Sarkar]{Pradip Das and Nabadwip Sarkar}
\address{Department of Mathematics, Raiganj University, Raiganj, West Bengal-733134, India.}
\email{pradipsmath@gmail.com}
\address{Amity School of Applied Sciences, Amity University Mumbai, Panvel, Navi Mumbai, Maharashtra-410206, India}
\email{nsarkar@mum.amity.edu}
\makeatletter
\@namedef{subjclassname@2020}{\textup{2020} Mathematics Subject Classification}
\makeatother


\begin{abstract}
Let $\mathcal{S}_{ex}^{\ast}$ denote the class of normalized analytic functions $f$ in the unit disk satisfying
\[
\frac{zf'(z)}{f(z)} \prec e^{\alpha z},\qquad 0<\alpha\le 1.
\]
We obtain sharp bounds for the initial inverse logarithmic coefficients $\Gamma_1$, $\Gamma_2$, and $\Gamma_3$. In particular, the bound for $|\Gamma_3|$ has a four-branch structure with transition values
\[
\alpha_\star \approx 0.542712,\qquad
\alpha_b \approx 0.679103,\qquad
\alpha_c \approx 0.691095.
\]
We also establish sharp upper and lower bounds for the successive coefficient difference
\[
|\Gamma_2|-|\Gamma_1|,
\]
for the inverse logarithmic Hankel determinant
\[
H_{2,1}\bigl(F_{f^{-1}}/2\bigr),
\]
and for the third-order Hermitian--Toeplitz determinant $T_{3,1}(f)$, whose sharp lower bound changes at
\[
\alpha_H=\frac{-1+\sqrt{241}}{15}\approx0.9683.
\]
Furthermore, we completely solve the sharp generalized Fekete--Szeg\H{o} problem for the functional
\[
|a_3-\lambda a_2^2|-\mu|a_2|.
\]
All estimates are sharp, and the corresponding extremal functions are explicitly constructed by using Carath\'eodory function representations.
\end{abstract}
\maketitle
\section{Introduction}

Let $\mathbb{D}:=\{z\in\mathbb{C}:|z|<1\}$ denote the open unit disk, and let $\mathcal{H}$ be the class of analytic functions in $\mathbb{D}$. A function $f\in\mathcal{H}$ belongs to the normalized class $\mathcal{A}$ if it has the form
\begin{equation}\label{eq1}
f(z)=z+\sum_{n=2}^{\infty}a_n z^n,\qquad z\in\mathbb{D}.
\end{equation}
that is, if it satisfies $f(0)=0$ and $f'(0)=1$. Let $\mathcal{S}\subset\mathcal{A}$ denote the usual subclass of univalent functions. We say that $f\in\mathcal{A}$ is starlike if its image $f(\mathbb{D})$ is starlike with respect to the origin. The class $\mathcal{S}^*$ of starlike functions is characterized analytically by
\[
f\in\mathcal{S}^*
\iff
\Re\left(\frac{zf'(z)}{f(z)}\right)>0,\qquad z\in\mathbb{D}.
\]
For the basic theory of univalent functions, we refer to the monographs of Duren \cite{PLD1} and Goodman \cite{G}.

Let $\Omega$ denote the class of Schwarz functions $\omega$ analytic in $\mathbb{D}$ with $\omega(0)=0$ and $|\omega(z)|<1$ for $z\in\mathbb{D}$. A function $f$ is subordinate to another analytic function $g$, written $f\prec g$, if there exists $\omega\in\Omega$ such that
\[
f(z)=g(\omega(z)),\qquad z\in\mathbb{D}.
\]
When $g$ is univalent, this is equivalent to the geometric condition
\[
f(0)=g(0),\qquad f(\mathbb{D})\subset g(\mathbb{D}).
\]
The notion of subordination is classical and plays a central role in geometric function theory; see, for example, Pommerenke \cite{Pommerenke1971}.

Ma and Minda \cite{MM} used subordination to give a unified treatment of many starlike and convex classes. They introduced the class
\[
\mathcal{S}^{*}(\phi)
:=
\left\{
f\in\mathcal{S}:
\frac{zf'(z)}{f(z)}\prec \phi(z),\ z\in\mathbb{D}
\right\},
\]
where $\phi$ is analytic and univalent in $\mathbb{D}$, $\Re\phi(z)>0$, $\phi(0)=1$, $\phi'(0)>0$, and $\phi(\mathbb{D})$ is symmetric about the real axis.

Coefficient problems for univalent functions have a long history, beginning with the Bieberbach conjecture, which was ultimately proved by de Branges \cite{LDB1}. Earlier contributions include Goluzin's distortion theorems \cite{G1}, Hayman's work on successive coefficients \cite{H}, and Grinspan's estimates for adjacent coefficients \cite{G2}. Among the many coefficient functionals studied in the literature, logarithmic coefficients have received considerable attention because of their close links with the geometry of univalent mappings and with the Schwarzian derivative. For $f\in\mathcal{S}$, the logarithmic coefficients $\gamma_n$ are defined by
\[
\log\frac{f(z)}{z}
=
2\sum_{n=1}^{\infty}\gamma_n z^n,\qquad z\in\mathbb{D}.
\]
Recent results on logarithmic coefficients and their differences have been obtained, for instance, by Obradovi\'c and Tuneski \cite{Obradovic2024}, Lecko and Partyka \cite{Lecko2023}, and Kumar and Cho \cite{Kumar2023}. Sharp bounds for successive coefficients and their differences have also been investigated by Leung \cite{Leung1978}, Li and Sugawa \cite{LiSugawa}, Peng and Obradovi\'c \cite{Peng2019}, and Arora and collaborators \cite{Arora2019,Arora2022,Arora2023}.

Recently, subclasses of starlike functions associated with exponential mappings have attracted much interest. The class $\mathcal{S}_e^*$ corresponding to the exponential function $e^z$ was introduced by Mendiratta et al.\ \cite{Mendiratta2015}. Later, Panja et al.\ \cite{Panja2026} studied classes governed by
\[
\phi(z)=e^{\alpha z},\qquad 0<\alpha\le1.
\]
Accordingly, a normalized function $f\in\mathcal{A}$ belongs to the class $\mathcal{S}_{ex}^{\ast}$ if it satisfies
\begin{equation}\label{subordination}
\frac{zf'(z)}{f(z)}\prec e^{\alpha z},\qquad z\in\mathbb{D}.
\end{equation}
Sharp bounds for Hermitian--Toeplitz determinants in the class $\mathcal{S}_e^*$ were obtained by Sarkar \cite{Sarkar}, while inverse logarithmic coefficient bounds were recently established by Das and Sarkar \cite{DasSarkar2026}. Higher-order Schwarzian derivative estimates for related classes were studied in \cite{Das2026a,Srivastava2026}.

The function $\phi(z)=e^{\alpha z}$ satisfies all standard Ma--Minda conditions: it is univalent, has positive real part in $\mathbb{D}$, is symmetric with respect to the real axis, and maps $\mathbb{D}$ onto a convex oval-shaped domain. As $\alpha$ increases from $0$ to $1$, the domain $\phi(\mathbb{D})$ expands continuously from a neighborhood of $w=1$ to a larger convex region. This behavior is illustrated in Figure~\ref{fig_domain}.

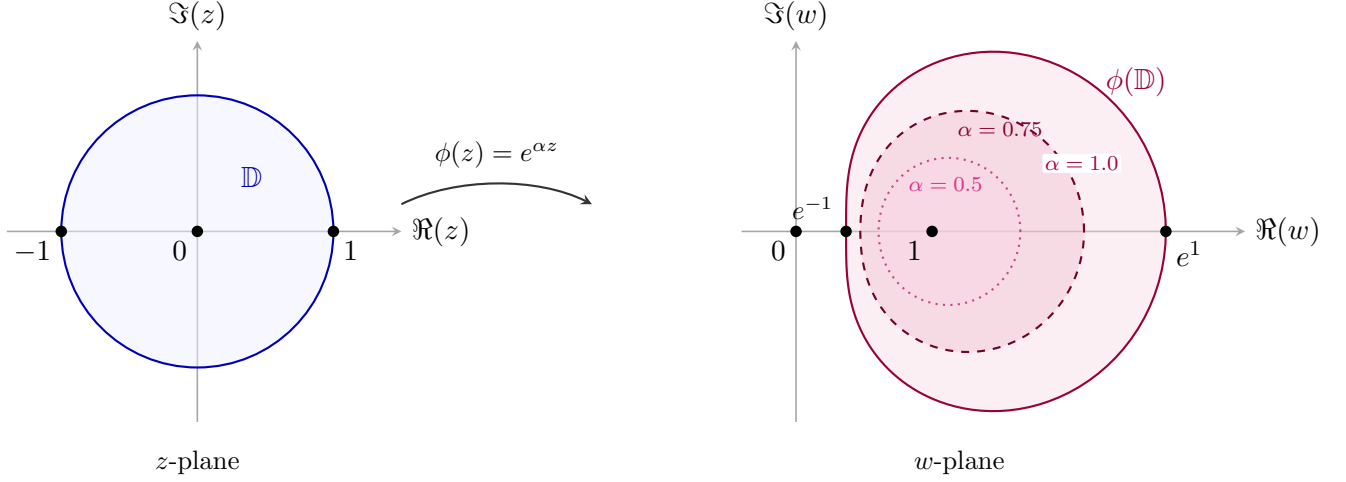
\begin{figure}[htbp]
\centering
\begin{tikzpicture}[scale=1.8, >=stealth]
    \begin{scope}[shift={(-2.2,0)}]
        \draw[->, gray!70, line width=0.6pt] (-1.4,0) -- (1.5,0) node[right, black] {$\Re(z)$};
        \draw[->, gray!70, line width=0.6pt] (0,-1.4) -- (0,1.4) node[above, black] {$\Im(z)$};
        
        \fill[blue!6, opacity=0.5] (0,0) circle (1);
        \draw[thick, blue!70!black] (0,0) circle (1);
        
        \node at (0,0) [below left] {$0$};
        \fill (0,0) circle (1.2pt);
        \fill (1,0) circle (1.2pt) node[below right] {$1$};
        \fill (-1,0) circle (1.2pt) node[below left] {$-1$};
        
        \node at (0.4,0.4) [blue!70!black, font=\bfseries] {$\mathbb{D}$};
        \node at (0,-1.7) [black] {\small $z$-plane};
    \end{scope}

    \draw[->, thick, black!80] (-0.7,0.2) .. controls (-0.3,0.4) and (0.3,0.4) .. (0.7,0.2) 
        node[midway, above=2pt, black, font=\small] {$\phi(z) = e^{\alpha z}$};

    \begin{scope}[shift={(2.2,0)}]
        \draw[->, gray!70, line width=0.6pt] (-0.4,0) -- (3.3,0) node[right, black] {$\Re(w)$};
        \draw[->, gray!70, line width=0.6pt] (0,-1.4) -- (0,1.4) node[above, black] {$\Im(w)$};
        
        \fill[purple!15, opacity=0.4] plot[domain=0:360, samples=180, variable=\t] 
            ({exp(1.0*cos(\t))*cos(1.0*sin(\t) r)}, {exp(1.0*cos(\t))*sin(1.0*sin(\t) r)});
        \draw[thick, purple!80!black] plot[domain=0:360, samples=180, variable=\t] 
            ({exp(1.0*cos(\t))*cos(1.0*sin(\t) r)}, {exp(1.0*cos(\t))*sin(1.0*sin(\t) r)});

        \fill[purple!25, opacity=0.4] plot[domain=0:360, samples=180, variable=\t] 
            ({exp(0.75*cos(\t))*cos(0.75*sin(\t) r)}, {exp(0.75*cos(\t))*sin(0.75*sin(\t) r)});
        \draw[thick, dashed, purple!60!black] plot[domain=0:360, samples=180, variable=\t] 
            ({exp(0.75*cos(\t))*cos(0.75*sin(\t) r)}, {exp(0.75*cos(\t))*sin(0.75*sin(\t) r)});

        \fill[magenta!20, opacity=0.5] plot[domain=0:360, samples=180, variable=\t] 
            ({exp(0.5*cos(\t))*cos(0.5*sin(\t) r)}, {exp(0.5*cos(\t))*sin(0.5*sin(\t) r)});
        \draw[thick, dotted, magenta!80!black] plot[domain=0:360, samples=180, variable=\t] 
            ({exp(0.5*cos(\t))*cos(0.5*sin(\t) r)}, {exp(0.5*cos(\t))*sin(0.5*sin(\t) r)});
            
        \node at (0,0) [below left] {$0$};
        \fill (0,0) circle (1.2pt);
        \fill (1,0) circle (1.2pt) node[below left] {$1$}; 
        
        \fill ({exp(1)},0) circle (1.2pt) node[below right] {$e^{1}$};
        \fill ({exp(-1)},0) circle (1.2pt) node[above left=1pt] {\small $e^{-1}$};
        
        \node at (2.1,0.5) [purple!90!black, font=\scriptsize, fill=white, inner sep=1pt, rounded corners=1pt] {$\alpha=1.0$};
        \node at (1.5,0.75) [purple!70!black, font=\scriptsize] {$\alpha=0.75$};
        \node at (1.1,0.35) [magenta!90!black, font=\scriptsize] {$\alpha=0.5$};
        
        \node at (2.5,1.1) [purple!80!black, font=\bfseries] {$\phi(\mathbb{D})$};
        \node at (1.2,-1.7) [black] {\small $w$-plane};
    \end{scope}
\end{tikzpicture}
\caption{Evolution of the image domain $\phi(\mathbb{D})$ for the parameter values $\alpha \in \{0.5, 0.75, 1.0\}$.}
\label{fig_domain}
\end{figure}

\subsection{Inverse logarithmic coefficients}

Let $f\in\mathcal{S}$. The inverse logarithmic coefficients $\Gamma_n$, introduced by Ponnusamy et al.\ \cite{Ponnusamy2018}, are defined via the expansion of the inverse function $f^{-1}$:
\[
F_{f^{-1}}(w)
:=
\log\frac{f^{-1}(w)}{w}
=
2\sum_{n=1}^{\infty}\Gamma_n w^n
\]
near $w=0$. These coefficients are the natural inverse analogue of the logarithmic coefficients and have been studied extensively because of their connections with the Bieberbach conjecture and the geometry of univalent mappings. Several authors have investigated inverse logarithmic coefficients and related Hankel and Toeplitz determinants; see, for example, \cite{MandalAhamed2024,MandalAhamed2026,MandalAhamedZaprawa2025}.

The first four inverse logarithmic coefficients can be expressed directly in terms of the Taylor coefficients $a_n$ of $f$ as
\begin{equation}\label{IG1}
\begin{cases}
\Gamma_1=-\dfrac{1}{2}a_2,\\[8pt]
\Gamma_2=-\dfrac{1}{2}a_3+\dfrac{3}{4}a_2^2,\\[8pt]
\Gamma_3=-\dfrac{1}{2}\left(a_4-4a_2a_3+\dfrac{10}{3}a_2^3\right),\\[8pt]
\Gamma_4=\dfrac{35}{8}a_2^4-\dfrac{15}{2}a_2^2a_3+\dfrac{5}{2}a_2a_4+\dfrac{5}{4}a_3^2-\dfrac{1}{2}a_5.
\end{cases}
\end{equation}

Ponnusamy et al.\ \cite{Ponnusamy2018} proved that, for the full class $\mathcal{S}$, the inverse logarithmic coefficients satisfy the sharp bound
\[
|\Gamma_n|
\le
\frac{1}{2n}\binom{2n}{n},\qquad n\in\mathbb{N},
\]
with equality for the Koebe function and its rotations. This result motivates the search for sharper bounds within important subclasses of $\mathcal{S}$.

Hankel and Toeplitz determinants involving logarithmic coefficients have also been intensively investigated. These determinants provide useful information about the structure of coefficient sequences and have applications in operator theory, signal processing, and random matrix theory. Kowalczyk and Lecko~\cite{12,15} introduced the logarithmic Hankel determinants
\[
H_{q,n}\left(F_f/2\right)
=
\begin{vmatrix}
\gamma_n & \gamma_{n+1} & \cdots & \gamma_{n+q-1}\\
\gamma_{n+1} & \gamma_{n+2} & \cdots & \gamma_{n+q}\\
\vdots & \vdots & \ddots & \vdots\\
\gamma_{n+q-1} & \gamma_{n+q} & \cdots & \gamma_{n+2(q-1)}
\end{vmatrix}.
\]
The corresponding inverse analogue
\[
H_{q,n}\bigl(F_{f^{-1}}/2\bigr)
=
\begin{vmatrix}
\Gamma_n & \Gamma_{n+1} & \cdots & \Gamma_{n+q-1}\\
\Gamma_{n+1} & \Gamma_{n+2} & \cdots & \Gamma_{n+q}\\
\vdots & \vdots & \ddots & \vdots\\
\Gamma_{n+q-1} & \Gamma_{n+q} & \cdots & \Gamma_{n+2(q-1)}
\end{vmatrix}
\]
was studied in~\cite{10,16}. Hankel determinant bounds for special subclasses have also been obtained by Sim et al.\ \cite{STZ} and Tang et al.\ \cite{TAA,TAH}.

Similarly, Hermitian--Toeplitz determinants, which are Hermitian analogues of classical Toeplitz determinants, have become important in coefficient problems for analytic functions. The theory of Toeplitz matrices goes back to Hartman and Wintner \cite{HW}; see also Kucerovsky et al.\ \cite{KMS}. In the context of univalent functions, Hermitian--Toeplitz determinants have been studied for numerous subclasses, including starlike, convex, bounded turning, and Janowski-type functions; see \cite{Ali,C,K3,KKD2023,KPR,KumarCho2021,KSC,GJSS2023,TKBSKS2025,UFBCL2025}.

Although logarithmic coefficient problems and Hankel determinant bounds have been studied for some exponential subclasses~\cite{Panja2026}, the inverse logarithmic functionals for the full parametric family $\mathcal{S}_{ex}^{\ast}$ remain largely open. In particular, the dependence on the parameter $\alpha$ produces a multi-branch structure that requires a careful case-by-case optimization.

The purpose of this paper is to give a unified treatment of sharp coefficient estimates for the class $\mathcal{S}_{ex}^{\ast}$ governed by $\phi(z)=e^{\alpha z}$, $0<\alpha\le1$. More precisely, we establish the following sharp results.

\begin{enumerate}
\item Sharp upper bounds for the initial inverse logarithmic coefficients $\Gamma_1$, $\Gamma_2$, and $\Gamma_3$. A key feature of our work is the four-branch structure of the bound for $|\Gamma_3|$, with transition values
\[
\alpha_* \approx 0.542712,\qquad
\alpha_b \approx 0.679103,\qquad
\alpha_c \approx 0.691095.
\]

\item Sharp upper and lower bounds for the successive coefficient difference
\[
|\Gamma_2|-|\Gamma_1|.
\]

\item Sharp bounds for the second-order inverse logarithmic Hankel determinant
\[
H_{2,1}\bigl(F_{f^{-1}}/2\bigr)
=
\Gamma_1\Gamma_3-\Gamma_2^2.
\]

\item Sharp upper and lower bounds for the third-order Hermitian--Toeplitz determinant $T_{3,1}(f)$. The lower bound changes at
\[
\alpha_H = \frac{-1+\sqrt{241}}{15} \approx 0.9683,
\]
and it reduces to the known value $-\frac{1}{15}$ when $\alpha=1$, thereby extending previous results for the class $\mathcal{S}_e^*$.

\item A complete solution of the generalized Fekete--Szeg\H{o} problem for
\[
|a_3-\lambda a_2^2|-\mu|a_2|,
\]
following the recent investigations of Lecko and Partyka \cite{Lecko2024} and Bulboac\u{a} et al.\ \cite{Bulboaca2025}.
\end{enumerate}

All estimates are sharp, and the corresponding extremal functions are constructed explicitly. Our approach relies on the Carath\'eodory parametrization of functions with positive real part, the Libera--Z{\l}otkiewicz parametrization \cite{LZ}, and on optimization lemmas due to Cho et al.\ \cite{C12}, Choi et al.\ \cite{CKS1}, Ma and Minda \cite{MM}, and Sim and Thomas \cite{SimThomas2020}. The extremal functions are obtained from the integral representation
\[
f(z)
=
z\exp\left(
\int_0^z \frac{1}{t}
\left[
\exp\left(
\alpha\frac{p(t)-1}{p(t)+1}
\right)
-1
\right]
dt
\right),
\]
where $p\in\mathcal{P}$ is a suitable Carath\'eodory function.

The paper is organized as follows. In Section~2, we recall the auxiliary lemmas needed in the proofs. Section~3 contains sharp bounds for the inverse logarithmic coefficients $\Gamma_1$, $\Gamma_2$, and $\Gamma_3$. In Section~4, we prove sharp bounds for the difference $|\Gamma_2|-|\Gamma_1|$. Section~5 is devoted to the second-order inverse logarithmic Hankel determinant. In Section~6, we obtain sharp upper and lower bounds for the third-order Hermitian--Toeplitz determinant. Section~7 contains the complete solution of the generalized Fekete--Szeg\H{o} problem. Finally, Section~8 summarizes the main results and outlines possible future directions.
\section{Auxiliary lemmas}

Let $\mathcal{P}$ denote the class of all analytic functions $p$ in the unit disk $\mathbb{D}$ satisfying $p(0) = 1$ and $\Re p(z) > 0$ for all $z \in \mathbb{D}$. Every $p \in \mathcal{P}$ admits the series representation
\begin{equation}\label{p1}
p(z) = 1 + \sum_{n=1}^{\infty} p_n z^n, \quad z \in \mathbb{D}.
\end{equation}
Functions in $\mathcal{P}$ are referred to as \emph{Carath\'{e}odory functions}. It is well-known that for $p \in \mathcal{P}$, the coefficients satisfy the sharp bound $|p_n| \le 2$ for all $n \ge 1$ (see \cite{PLD1}).

Now we recall the following well-known lemma due to Cho et al. \cite{C12}.
\begin{lem}\label{L1} \cite[Lemma 2.4]{C12} If $p\in\mathcal{P}$ is of the form (\ref{p1}), then
\begin{equation}\label{c1}p_1 =2\tau_1,\end{equation}
\begin{equation}\label{c2} p_2=2\tau_1^2 + 2(1 - |\tau_1|^2)\tau_2\end{equation}
and
\begin{equation}\label{c3} p_3 = 2\tau_1^3+4(1-|\tau_1|^2)\tau_1\tau_2 - 2(1 - |\tau_1|^2)\ol{\tau_1}\tau_2^2 + 2(1 - \tau_1^2)(1 - |\tau_2|^2)\tau_3
\end{equation}
for some $\tau_1, \tau_2, \tau_3 \in\mathbb{\ol D}:= \{z \in \mathbb{C}: |z| \le 1 \}$.
For $ \tau_1 \in \mathbb{T} := \{ z \in \mathbb{C} : |z| = 1 \} $, there is a unique function $ p \in \mathcal{P} $ with $ p_1 $ as in (\ref{c1}), namely,
\[
p(z) = \frac{1 + \tau_1 z}{1 - \tau_1 z}, \quad z \in \mathbb{D}.
\]
For $ \tau_1 \in \mathbb{D} $ and $ \tau_2 \in \mathbb{T} $, there is a unique function $ p \in \mathcal{P} $ with $ p_1 $ and $ p_2 $ as in (\ref{c1}) and (\ref{c2}), namely,
\[
p(z) = \frac{1 + (\ol \tau_1 \tau_2 + \tau_1) z + \tau_2 z^2}{1 + (\ol \tau_1 \tau_2 - \tau_1) z - \tau_2 z^2}, \quad z \in \mathbb{D}.
\]
For $ \tau_1, \tau_2 \in \mathbb{D} $ and $ \tau_3 \in \mathbb{T} $, there is a unique function $ p \in \mathcal{P} $ with $ p_1 $, $ p_2 $, and $ p_3 $ as in (\ref{c1}-\ref{c3}), namely,
\[
p(z) = \frac{1 + (\ol\tau_2 \tau_3 + \ol\tau_1 \tau_2 + \tau_1)z+(\ol\tau_1\tau_3+\tau_1\ol\tau_2\tau_3+\tau_2)z^2+\tau_3z^3}{1+(\ol\tau_2\tau_3+\ol\tau_1\tau_2-\tau_1)z+(\ol\tau_1\tau_3-\tau_1\ol\tau_2\tau_3-\tau_2)z^2-\tau_3z^3},\;\;z\in\mathbb{D}
\]
\end{lem}

Following a well-known result due to Choi et al. \cite{CKS1}.
\begin{lem}\label{L2}\cite{CKS1} Let $A$, $B$, $C$ be real numbers and let
\[Y(A, B, C):= \max\limits_{z\in \ol{\mathbb{D}}}\left\lbrace |A+Bz+Cz^2|+1-|z|^2\right\rbrace.\]
\begin{enumerate} 
\item[(i)] If $AC\ge 0$, then
\[Y(A, B, C) =
\begin{cases}
|A|+|B|+|C|, & \text{if}\;\;\; |B|\ge 2(1-|C|), \\
1+|A|+\frac{B^2}{4(1-|C|)}, &\text{if}\;\;\; |B|<2(1-|C|).
\end{cases}
\]
\item[(ii)] If $AC<0$, then 
\[Y(A,B,C)=
\begin{cases}
1-|A|+\frac{B^2}{4(1-|C|)}, &\text{if}\;\;\; -4AC(C^{-2}-1) \le B^2\; \text{and}\; |B|<2(1-|C|), \\
1+|A|+\frac{B^2}{4(1+|C|)}, &\text{if}\;\;\; B^2<\min\left\{4(1+|C|)^2, -4AC(C^{-2}-1) \right\}, \\
R(A,B,C), &\text{otherwise},
\end{cases}
\]
where
\[R(A,B,C):=
\begin{cases}
|A|+|B|-|C|, & \text{if}\;\;\; |C|(|B|+4|A|) \le |AB|, \\
-|A|+|B|+|C|, & \text{if}\;\;\; |AB|\le |C|(|B|-4|A|), \\
(|C|+|A| )\sqrt{1-\frac{B^2}{4AC}}, &\text{otherwise}.
\end{cases}
\]
\end{enumerate} 
\end{lem}

\begin{lem}\label{L3} \cite{MM}
Let $p \in \mathcal{P}$ be given by \eqref{p1}. Then
\[
\left| p_2 - v p_1^2 \right| \le 
\begin{cases}
-4v + 2, & v < 0, \\
2, & 0 \le v \le 1, \\
4v - 2, & v > 1.
\end{cases}
\]
Moreover, for $v < 0$ or $v > 1$, equality holds if and only if
\[
h(z) = \frac{1+z}{1-z} \quad \text{or one of its rotations}.
\]
For $0 < v < 1$, equality holds if and only if
\[
h(z) = \frac{1+z^2}{1-z^2} \quad \text{or one of its rotations}.
\]
\end{lem}

\begin{lem}\label{L6}\cite{SimThomas2020}
Let $J, K,$ and $L$ be numbers such that $J \ge 0$, $K \in \mathbb{C}$, and $L \in \mathbb{R}$. 
Let $p \in \mathcal{P}$ be of the form (\ref{p1}) and define a function by
\[
\Phi(p_1,p_2) = \big| K p_1^2 + L p_2 \big| - \big| J p_1 \big|.
\]
Then 
\[
\Phi(p_{1}, p_{2}) \le 
\begin{cases}
|4K + 2L| - 2J, & \text{if } |2K + L| \ge |L| + J, \\[6pt]
2|L|, & \text{otherwise.}
\end{cases}
\]
and
\[
- \Phi(p_1,p_2) \le 
\begin{cases}
2J - M, & \text{when } J \ge M + 2|L|, \\[6pt]
2J \sqrt{\dfrac{ 2|L|}{M + 2|L|}}, & \text{when } J^2 \le 2|L|(M + 2|L|), \\[10pt]
2|L| +\dfrac{ J^2}{M + 2|L|}, & \text{otherwise}
\end{cases}
\]
where $M=|4K+2L|$.
\end{lem}

We need the following elementary sign lemma to prove Theorem \eqref{T3}:

\begin{lem}\label{lem:Hsign}
For $0<\alpha\le 1$, let $\widetilde H$ be defined by \eqref{Htilde}, and set
\[
d=\alpha^2,\qquad d_0:=\frac{2688}{3395}.
\]
Then $\widetilde H$ has no zero in $(0,1)$ at which it changes sign from positive to negative. More precisely:
\begin{itemize}
\item[(i)] if $0<d\le d_0$, then $\widetilde H(y)<0$ for all $y\in[0,1)$;
\item[(ii)] if $d_0<d\le1$, then $\widetilde H$ is strictly increasing on $[0,1]$, with
\[
\widetilde H(0)<0<\widetilde H(1),
\]
and $\widetilde H$ has a unique zero in $(0,1)$, across which its sign changes from negative to positive.
\end{itemize}
\end{lem}

\begin{proof}[Proof of Lemma]
Write
\[
\widetilde H(y)=A y^3+B y^2+C y-7344,
\]
where
\[
A=384-1120d,\qquad B=528-420d,\qquad C=15120d-4320.
\]
Also
\[
\widetilde H(0)=-7344<0,
\qquad
\widetilde H(1)=13580d-10752
=4(3395d-2688).
\]

We split the proof into two cases.

\paragraph{Case 1: $0<d\le \frac{12}{35}$.}
Then $A\ge0$ and
\[
\widetilde H(1)\le 4\left(3395\cdot \frac{12}{35}-2688\right)
=4(1164-2688)<0.
\]

If $0<d\le \frac27$, then $C\le0$ and $B>0$. Hence $\widetilde H'$ has one negative and one positive root, so $\widetilde H$ first decreases and then increases. Since $\widetilde H(0)<0$ and $\widetilde H(1)<0$, it follows that $\widetilde H(y)<0$ for all $y\in[0,1]$.

If $\frac27<d\le \frac{12}{35}$, then $A>0$, $B>0$, and $C>0$. Thus $\widetilde H'(y)>0$ for all $y\ge0$. Hence $\widetilde H$ is strictly increasing on $[0,1]$, and since $\widetilde H(1)<0$, we again have $\widetilde H(y)<0$ for all $y\in[0,1]$.

\paragraph{Case 2: $\frac{12}{35}<d\le1$.}
Then
\[
A<0,\qquad B>0,\qquad C>0.
\]
Moreover,
\[
\widetilde H'(0)=C=15120d-4320>0,
\]
and
\[
\widetilde H'(1)=3A+2B+C=10920d-2112>0.
\]
Since $\widetilde H'$ is a concave down quadratic and is positive at both endpoints $0$ and $1$, it is positive on the whole interval $[0,1]$. Therefore $\widetilde H$ is strictly increasing on $[0,1]$.

If $\frac{12}{35}<d\le d_0$, then $\widetilde H(1)\le0$, so $\widetilde H(y)<0$ for all $y\in[0,1)$.

If $d_0<d\le1$, then $\widetilde H(0)<0<\widetilde H(1)$. Hence $\widetilde H$ has a unique zero in $(0,1)$, and its sign changes from negative to positive.

Thus in all cases $\widetilde H$ never changes sign from positive to negative on $(0,1)$.
\end{proof}

\section{Sharp Bounds for the Inverse Logarithmic Coefficients}

\begin{theo}\label{T1}
Let $0<\alpha\le 1$. If $f\in \mathcal{S}_{ex}^{\ast}$, then the inverse logarithmic coefficients $\Gamma_n$ $(n=1,2,3)$ satisfy the following sharp inequalities:
\[
|\Gamma_1|\le \frac{\alpha}{2},
\]
\[
|\Gamma_2|\le
\begin{cases}
\dfrac{\alpha}{4}, & 0<\alpha\le \dfrac23,\\[6pt]
\dfrac{3\alpha^2}{8}, & \dfrac23<\alpha\le1,
\end{cases}
\]
and
\begin{equation}\label{ND3}
|\Gamma_3|\le 
\begin{cases}
\dfrac{\alpha}{6}, & 0<\alpha\le \alpha_\star,\\[8pt]
\dfrac{\alpha(7\alpha+2)}{18}
\sqrt{\dfrac{2(7\alpha+2)}{29\alpha^2+42\alpha+12}},
& \alpha_\star<\alpha\le \alpha_b,\\[8pt]
\dfrac{29\alpha(49\alpha^2-16)^{3/2}}
{2232\sqrt{29\alpha^2-12}},
& \alpha_b<\alpha<\alpha_c,\\[8pt]
\dfrac{29\alpha^3}{72}, & \alpha_c\le \alpha\le 1,
\end{cases}
\end{equation}
where
\[
\begin{aligned}
\alpha_\star &\approx 0.542712
\quad \text{is the unique positive root of } 686\alpha^3+327\alpha^2-210\alpha-92=0,\\[2mm]
\alpha_b &\approx 0.679103
\quad \text{is the unique positive root of } 1421\alpha^3+812\alpha^2-712\alpha-336=0,\\[2mm]
\alpha_c &=\sqrt{\frac{928}{1943}}\approx 0.691095.
\end{aligned}
\]
All bounds are sharp for every $\alpha\in(0,1]$.
\end{theo}

\begin{proof}
For $f\in\mathcal{S}_{ex}^{\ast}$, write
\[
\frac{zf'(z)}{f(z)}=e^{\alpha\omega(z)},
\]
where $\omega\in\Omega$. Taking
\[
\omega(z)=\frac{p(z)-1}{p(z)+1},
\qquad
p(z)=1+\sum_{n=1}^{\infty} p_n z^n\in\mathcal{P},
\]
and comparing coefficients gives
\begin{align}
a_2&=\frac{\alpha p_1}{2},\label{ND7}\\
a_3&=\frac{\alpha p_2}{4}
+\frac{3\alpha^2-2\alpha}{16}p_1^2,\label{ND8}\\
a_4&=\frac{\alpha p_3}{6}
+\frac{5\alpha^2-4\alpha}{24}p_1p_2
+\frac{17\alpha^3-30\alpha^2+12\alpha}{288}p_1^3.\label{ND9}
\end{align}

The inverse logarithmic coefficients satisfy
\[
\Gamma_1=-\frac12 a_2,\qquad
\Gamma_2=-\frac12 a_3+\frac34 a_2^2,\qquad
\Gamma_3=-\frac12\left(a_4-4a_2a_3+\frac{10}{3}a_2^3\right).
\]
Hence
\[
|\Gamma_1|=\frac{\alpha}{4}|p_1|\le \frac{\alpha}{2},
\]
and
\[
\Gamma_2
=-\frac{\alpha}{8}p_2
+\frac{3\alpha^2+2\alpha}{32}p_1^2.
\]
Thus
\[
|\Gamma_2|
=\frac{\alpha}{8}
\left|p_2-\frac{3\alpha+2}{4}p_1^2\right|.
\]
Applying Lemma~\ref{L3}, we obtain
\[
|\Gamma_2|\le \frac{\alpha}{4}
\]
for $0<\alpha\le \frac23$, and
\[
|\Gamma_2|\le \frac{3\alpha^2}{8}
\]
for $\frac23<\alpha\le 1$.

We now estimate $|\Gamma_3|$. Substituting \eqref{ND7}--\eqref{ND9} into the expression for $\Gamma_3$ gives
\begin{equation}\label{ND11}
\Gamma_3
=
-\frac{\alpha}{12}p_3
+\frac{7\alpha^2+4\alpha}{48}p_1p_2
-\frac{29\alpha^3+42\alpha^2+12\alpha}{576}p_1^3.
\end{equation}

By rotational invariance of $\mathcal{P}$, we may assume $p_1\in[0,2]$. Put
\[
p_1=2\tau_1,\qquad \tau_1\in[0,1].
\]
Using Lemma~\ref{L1}, write
\[
p_2=2\tau_1^2+2(1-\tau_1^2)\tau_2,
\]
and
\[
p_3
=
2\tau_1^3
+4(1-\tau_1^2)\tau_1\tau_2
-2(1-\tau_1^2)\tau_1\tau_2^2
+2(1-\tau_1^2)(1-|\tau_2|^2)\tau_3,
\]
where $\tau_2,\tau_3\in\overline{\mathbb{D}}$.

Then
\begin{equation}\label{ND12}
\begin{aligned}
|\Gamma_3|
=
\frac{\alpha}{12}\Bigg|
&-\frac{29\alpha^2}{6}\tau_1^3
+7\alpha(1-\tau_1^2)\tau_1\tau_2\\
&+2(1-\tau_1^2)\tau_1\tau_2^2
-2(1-\tau_1^2)(1-|\tau_2|^2)\tau_3
\Bigg|.
\end{aligned}
\end{equation}

If $\tau_1=1$, then $p_1=p_2=p_3=2$, and
\[
|\Gamma_3|=\frac{29\alpha^3}{72}.
\]
If $\tau_1=0$, then
\[
|\Gamma_3|=\frac{\alpha}{12}|-2\tau_3|
\le \frac{\alpha}{6}.
\]

Assume $0<\tau_1<1$. Applying the triangle inequality to \eqref{ND12} gives
\begin{equation}\label{ND13}
|\Gamma_3|
\le
\frac{\alpha(1-\tau_1^2)}{6}
\left(
|A+B\tau_2+C\tau_2^2|+1-|\tau_2|^2
\right),
\end{equation}
where
\[
A=\frac{29\alpha^2\tau_1^3}{12(1-\tau_1^2)}>0,
\qquad
B=-\frac{7\alpha}{2}\tau_1,
\qquad
C=-\tau_1.
\]
Since $AC<0$, we apply Lemma~\ref{L2}(ii).

Define
\[
t_1=\frac{4}{7\alpha+4},
\qquad
t_d=\sqrt{\frac{84}{203\alpha^2+232\alpha+84}}.
\]

\paragraph{Subcase (a).}
The conditions are
\[
|B|<2(1-|C|),\qquad -4AC(C^{-2}-1)\le B^2.
\]
The first inequality is equivalent to $\tau_1<t_1$, while the second is always true because
\[
\frac{29}{3}<\frac{49}{4}.
\]
Thus Subcase (a) applies for $0<\tau_1<t_1$. In this case,
\[
Y_a(\tau_1)
=
1-A+\frac{B^2}{4(1-|C|)}.
\]
Hence
\[
F_a(\tau_1)
:=
\frac{\alpha(1-\tau_1^2)}{6}Y_a(\tau_1)
=
\frac{\alpha}{6}
+
\left(\frac{49\alpha^3}{96}-\frac{\alpha}{6}\right)\tau_1^2
+
\frac{31\alpha^3}{288}\tau_1^3.
\]
Differentiating,
\[
F_a'(\tau_1)
=
\frac{\alpha\tau_1}{96}
\left(
31\alpha^2\tau_1+98\alpha^2-32
\right).
\]
Thus $F_a$ has no interior maximum on $(0,t_1)$. Therefore
\[
\sup_{0<\tau_1<t_1} F_a(\tau_1)
=
\max\left\{
\frac{\alpha}{6},
F_0(\alpha)
\right\},
\]
where
\[
F_0(\alpha)
=
\frac{\alpha^2(1029\alpha^2+1238\alpha+336)}
{9(7\alpha+4)^3}.
\]

\paragraph{Subcase (b).}
The required condition
\[
B^2<-4AC(C^{-2}-1)
\]
becomes
\[
\frac{49}{4}<\frac{29}{3},
\]
which is impossible. Hence Subcase (b) is inactive.

\paragraph{Subcase (c).}
The conditions are
\[
|B|\ge 2(1-|C|),\qquad |C|(|B|+4|A|)\le |AB|.
\]
The first gives $\tau_1\ge t_1$. The second becomes
\[
(203\alpha^2-232\alpha+84)\tau_1^2\ge 84.
\]
For $0<\alpha\le1$,
\[
203\alpha^2-232\alpha+84<84,
\]
so this inequality is impossible. Hence Subcase (c) is inactive.

\paragraph{Subcase (d).}
The conditions are
\[
|B|\ge 2(1-|C|),\qquad |AB|\le |C|(|B|-4|A|).
\]
The first gives $\tau_1\ge t_1$. The second gives $\tau_1\le t_d$. Since $t_1<t_d$ for all $\alpha\in(0,1]$, Subcase (d) applies on $[t_1,t_d]$. In this case,
\[
Y_d(\tau_1)
=
-|A|+|B|+|C|.
\]
Therefore,
\[
F_d(\tau_1)
:=
\frac{\alpha(1-\tau_1^2)}{6}Y_d(\tau_1)
=
-\frac{29\alpha^3}{72}\tau_1^3
+
\frac{\alpha(7\alpha+2)}{12}\tau_1(1-\tau_1^2).
\]
Differentiating,
\[
F_d'(\tau_1)
=
\frac{\alpha}{24}
\left(
2(7\alpha+2)
-(29\alpha^2+42\alpha+12)\tau_1^2
\right).
\]
Hence the critical point is
\[
\tau_d^*
=
\sqrt{
\frac{2(7\alpha+2)}
{29\alpha^2+42\alpha+12}
}.
\]
Since $F_d''(\tau_1)<0$ for $\tau_1>0$, $F_d$ is concave.

Define $\beta$ and $\alpha_b$ by
\[
\tau_d^*(\beta)=t_1(\beta),\qquad
\tau_d^*(\alpha_b)=t_d(\alpha_b).
\]
Equivalently,
\[
343\beta^3+258\beta^2-112\beta-64=0,
\]
so $\beta\approx 0.529607$, and
\[
1421\alpha_b^3+812\alpha_b^2-712\alpha_b-336=0,
\]
so $\alpha_b\approx 0.679103$.
Thus the maximum of $F_d$ on $[t_1,t_d]$ is
\[
F_d^{\max}(\alpha)
=
\begin{cases}
F_0(\alpha), & 0<\alpha\le \beta,\\[6pt]
F_d(\tau_d^*), & \beta<\alpha\le \alpha_b,\\[6pt]
F_d(t_d), & \alpha_b<\alpha\le1.
\end{cases}
\]
A direct computation gives
\[
F_d(\tau_d^*)
=
\frac{\alpha(7\alpha+2)}{18}
\sqrt{
\frac{2(7\alpha+2)}
{29\alpha^2+42\alpha+12}
}.
\]

\paragraph{Subcase (e).}
This subcase applies for $t_d<\tau_1<1$. By Lemma~\ref{L2},
\[
F_e(\tau_1)
=
\frac{\alpha}{72\sqrt{116}}
\left(
12+(29\alpha^2-12)\tau_1^2
\right)
\sqrt{147-31\tau_1^2}.
\]
Set
\[
H(x)=[F_e(\sqrt{x})]^2,\qquad x=\tau_1^2.
\]
Then
\[
H(x)
=
\frac{\alpha^2}{72^2\cdot 116}
\left[
12+(29\alpha^2-12)x
\right]^2
(147-31x).
\]
Differentiating,
\[
H'(x)
=
\frac{\alpha^2}{72^2\cdot116}
\left[
12+(29\alpha^2-12)x
\right]
S(x),
\]
where
\[
S(x)
=
8526\alpha^2-3900-(2697\alpha^2-1116)x.
\]
The first factor is positive for $0\le x\le1$, so the monotonicity of $H$ is controlled by $S(x)$.

The zero of $S$ is
\[
x_e^*
=
\frac{2842\alpha^2-1300}{899\alpha^2-372}.
\]
Also,
\[
S(1)
=
3(1943\alpha^2-928),
\]
so $\alpha_c$ is defined by
\[
\alpha_c^2=\frac{928}{1943}.
\]
A direct calculation gives
\[
S(t_d^2)
=
\frac{
(1218\alpha+696)
(1421\alpha^3+812\alpha^2-712\alpha-336)
}
{203\alpha^2+232\alpha+84}.
\]
Therefore:
\begin{itemize}
\item for $0<\alpha\le\alpha_b$, $S(t_d^2)\le0$ and $S(1)<0$, so $S(x)\le0$ on $[t_d^2,1]$; hence $F_e$ is decreasing and its maximum is at $t_d$;
\item for $\alpha_b<\alpha<\alpha_c$, $S(t_d^2)>0$ and $S(1)<0$, so $x_e^*\in(t_d^2,1)$; hence the maximum of $F_e$ occurs at $x_e^*$;
\item for $\alpha_c\le\alpha\le1$, $S(t_d^2)\ge0$ and $S(1)\ge0$, so $S(x)\ge0$ on $[t_d^2,1]$; hence $F_e$ is increasing and its maximum is at $1$.
\end{itemize}
Thus
\[
\widetilde F_e(\alpha)
:=
\max_{t_d\le \tau_1\le1} F_e(\tau_1)
=
\begin{cases}
F_d(t_d), & 0<\alpha\le\alpha_b,\\[6pt]
F_e(\sqrt{x_e^*}), & \alpha_b<\alpha<\alpha_c,\\[6pt]
F_e(1)=\dfrac{29\alpha^3}{72}, & \alpha_c\le\alpha\le1.
\end{cases}
\]
Moreover,
\[
F_e(\sqrt{x_e^*})
=
\frac{29\alpha(49\alpha^2-16)^{3/2}}
{2232\sqrt{29\alpha^2-12}}.
\]

\subsection*{Phase transition analysis}

The global maximum is
\[
M(\alpha)
=
\max\left\{
\frac{\alpha}{6},
F_0(\alpha),
F_d^{\max}(\alpha),
\widetilde F_e(\alpha),
\frac{29\alpha^3}{72}
\right\}.
\]

\paragraph{Region I: $0<\alpha\le\alpha_\star$.}
Let
\[
R(\alpha)
=
1029\alpha^3+712\alpha^2-336\alpha-192.
\]
Then
\[
\frac{\alpha}{6}-F_0(\alpha)
=
-\frac{\alpha R(\alpha)}{18(7\alpha+4)^3}.
\]
Since $R(0)=-192<0$ and $R(0.55)>0$, the polynomial $R$ has exactly one positive root, $\alpha_0\approx 0.542829$. Since $\beta\approx 0.529607<\alpha_0$,
\[
\frac{\alpha}{6}>F_0(\alpha)
\]
for $0<\alpha\le\beta$.

For $\beta<\alpha\le\alpha_\star$,
\[
F_0(\alpha)=F_d(t_1)\le F_d(\tau_d^*).
\]
Let
\[
Q(\alpha)
=
686\alpha^3+327\alpha^2-210\alpha-92.
\]
Then $\alpha_\star$ is the unique positive root of $Q$. Also,
\[
\left(\frac{\alpha}{6}\right)^2
-
F_d(\tau_d^*)^2
=
-\frac{\alpha^2 Q(\alpha)}
{324(29\alpha^2+42\alpha+12)}.
\]
Since $Q(\alpha)<0$ for $0<\alpha<\alpha_\star$,
\[
\frac{\alpha}{6}\ge F_d(\tau_d^*).
\]
Finally,
\[
\frac{\alpha}{6}
-
\frac{29\alpha^3}{72}
=
\frac{\alpha}{72}(12-29\alpha^2)>0
\]
for $\alpha\le\alpha_\star$. Therefore
\[
M(\alpha)=\frac{\alpha}{6}
\]
for $0<\alpha\le\alpha_\star$.

\paragraph{Region II: $\alpha_\star<\alpha\le\alpha_b$.}
For $\alpha>\alpha_\star$, $Q(\alpha)>0$, so
\[
F_d(\tau_d^*)>\frac{\alpha}{6}.
\]
Also,
\[
F_d(\tau_d^*)^2-F_0(\alpha)^2>0.
\]
Thus $F_d(\tau_d^*)$ dominates $F_0(\alpha)$.

For $\alpha\le\alpha_b$,
\[
\widetilde F_e(\alpha)=F_d(t_d)\le F_d(\tau_d^*).
\]
Finally,
\[
F_d(\tau_d^*)^2
-
\left(\frac{29\alpha^3}{72}\right)^2
>0
\]
for $\alpha\le\alpha_b$. Therefore
\[
M(\alpha)=F_d(\tau_d^*)
\]
for $\alpha_\star<\alpha\le\alpha_b$.

\paragraph{Region III: $\alpha_b<\alpha<\alpha_c$.}
For this range,
\[
F_d(\tau_d^*)>F_0(\alpha),\qquad
F_d(\tau_d^*)>\frac{\alpha}{6}.
\]
A direct calculation gives
\[
\left(
\frac{F_e(\sqrt{x_e^*})}
{F_d(\tau_d^*)}
\right)^2-1
=
\frac{
(1421\alpha^2+434\alpha-340)
(1421\alpha^3+812\alpha^2-712\alpha-336)^2
}
{30752(7\alpha+2)^3(29\alpha^2-12)}
>0.
\]
Thus
\[
F_e(\sqrt{x_e^*})>F_d(\tau_d^*).
\]
Also,
\[
F_e(\sqrt{x_e^*})^2
-
\left(\frac{29\alpha^3}{72}\right)^2
=
\frac{
841\alpha^2(5\alpha^2-1)(67\alpha^2-32)^2
}
{1245456(29\alpha^2-12)}
>0.
\]
Therefore
\[
M(\alpha)=F_e(\sqrt{x_e^*})
\]
for $\alpha_b<\alpha<\alpha_c$.

\paragraph{Region IV: $\alpha_c\le\alpha\le1$.}
For $\alpha\ge\alpha_c$, $S(x)\ge0$ on $[t_d^2,1]$, so
\[
\widetilde F_e(\alpha)=F_e(1)=\frac{29\alpha^3}{72}.
\]
Also,
\[
\frac{29\alpha^3}{72}\ge \frac{\alpha}{6},\qquad
\frac{29\alpha^3}{72}\ge F_0(\alpha),
\]
and
\[
\left(\frac{29\alpha^3}{72}\right)^2
-
F_d(t_d)^2
\ge0.
\]
Hence
\[
M(\alpha)=\frac{29\alpha^3}{72}
\]
for $\alpha_c\le\alpha\le1$.

Combining the four regions proves \eqref{ND3}.

\subsection*{Sharpness and extremal functions}

The extremal functions are obtained from
\begin{equation}\label{pra27}
f(z)
=
z\exp\left(
\int_0^z \frac{1}{t}
\left[
\exp\left(
\alpha\frac{p(t)-1}{p(t)+1}
\right)
-1
\right]
dt
\right).
\end{equation}

\paragraph{Region I.}
Choose $\tau_1=0$, $\tau_2=0$, $\tau_3=1$. Then $p_1=p_2=0$, $p_3=2$, and
\[
|\Gamma_3|=\frac{\alpha}{6}.
\]
The extremal function is
\[
f_{\mathrm{tri}}(z)
=
z\exp\left(
\int_0^z \frac{e^{\alpha t^3}-1}{t}\,dt
\right).
\]

\paragraph{Region II.}
Choose $\tau_1=\tau_d^*$ and $\tau_2=1$. Since $|\tau_2|=1$, the $\tau_3$-term vanishes. Lemma~\ref{L1} guarantees the existence of the corresponding $p\in\mathcal{P}$, and equality is attained.

\paragraph{Region III.}
Choose $\tau_1=\sqrt{x_e^*}$ and $\tau_2=e^{i\theta}$, where
\[
\cos\theta
=
\frac{7\left[12-(29\alpha^2+12)x_e^*\right]}
{232\alpha x_e^*}.
\]
For $\alpha_b<\alpha<\alpha_c$, we have $t_d^2<x_e^*<1$. Since
\[
1-\cos\theta
=
\frac{
(203\alpha^2+232\alpha+84)x_e^*-84
}
{232\alpha x_e^*}
>0,
\]
and
\[
1+\cos\theta
=
\frac{
84-(203\alpha^2-232\alpha+84)x_e^*
}
{232\alpha x_e^*}
>0,
\]
we get $|\cos\theta|<1$. Thus such $\theta$ exists, and equality is attained.

\paragraph{Region IV.}
Choose $\tau_1=1$. Then $p_1=p_2=p_3=2$, and
\[
|\Gamma_3|=\frac{29\alpha^3}{72}.
\]
The extremal function is the Koebe-type function
\[
f_{\mathrm{Koebe}}(z)
=
z\exp\left(
\int_0^z \frac{e^{\alpha t}-1}{t}\,dt
\right).
\]

This completes the proof.
\end{proof}
\section{Bounds for the differences of inverse logarithmic coefficients}
The Bieberbach conjecture was proved in 1985 by de Branges~\cite{LDB1}, who established that the Taylor coefficients of any function $f \in \mathcal{S}$ of the form~\eqref{eq1} satisfy the inequality $|a_n| \le n$ for all $n \ge 2$. For the starlike subclass $\mathcal{S}^\ast$, the coefficients are bounded by $n$, a result proved by Leung~\cite{Leung1978}. Within the class of convex functions, Li and Sugawa~\cite{LiSugawa} investigated the difference of adjacent coefficients, while further results on successive differences were obtained in~\cite{Peng2019, Arora2019, Arora2022, Arora2023}.

Recently, attention has shifted toward the difference of logarithmic coefficients. In 2024, Lecko and Partyka~\cite{Lecko2023} utilized the Loewner differential equation to establish sharp bounds for $|\gamma_2| - |\gamma_1|$ within the class $\mathcal{S}$. Following this direction, Kumar and Cho~\cite{Kumar2023} derived sharp estimates for the coefficient difference for certain subclasses of $\mathcal{S}$. Motivated by these investigations, this section focuses on establishing sharp upper and lower bounds for the difference of the initial inverse logarithmic coefficients, namely $|\Gamma_2| - |\Gamma_1|$, for functions belonging to the subclass $\mathcal{S}_{ex}^{\ast}$.

\begin{theo}\label{T2}
Let $0 < \alpha \le 1$. If $f \in \mathcal{S}_{ex}^{\ast}$ and the inverse logarithmic coefficients $\Gamma_n$ ($n = 1,2$) are defined by \eqref{IG1}, then
\begin{equation}\label{ND15}
-\frac{\alpha}{2}\sqrt{\frac{2}{3\alpha + 2}} \le |\Gamma_2|-|\Gamma_1| \le \frac{\alpha}{4}.
\end{equation}
These inequalities are sharp for all $\alpha \in (0,1]$.
\end{theo}

\begin{proof}
Let $f \in \mathcal{S}_{ex}^{\ast}$. Using the expressions for $\Gamma_1$ and $\Gamma_2$ in terms of the parameters $p_1$ and $p_2$, the difference is given by:
\begin{align}
|\Gamma_2| - |\Gamma_1|
&= \left| -\frac{\alpha}{8}p_2 + \frac{3\alpha^2+2\alpha}{32}p_1^2 \right| - \left| -\frac{\alpha}{4}p_1 \right| \nonumber \\
&= \frac{\alpha}{32}\left( |(3\alpha+2)p_1^2 - 4p_2| - 8|p_1| \right) \nonumber \\
&= \frac{\alpha}{32}\Phi(p_1, p_2).
\label{ND16}
\end{align}
The expression $\Phi(p_1, p_2) = |Kp_1^2 + Lp_2| - |Jp_1|$ satisfies the conditions of Lemma \ref{L6} with parameters $K = 3\alpha+2$, $L = -4$, and $J = 8$. It follows that $M = |4K + 2L| = 12\alpha$ and $|2K + L| = 6\alpha$.

To determine the upper bound, we check the condition $|2K + L| \ge |L| + J$, which requires $6\alpha \ge 12$, or $\alpha \ge 2$. Since $0 < \alpha \le 1$, this condition is not satisfied, and the alternative branch of Lemma \ref{L6} applies, giving $\Phi(p_1, p_2) \le 2|L| = 8$. Substituting this value into \eqref{ND16} yields the upper bound $|\Gamma_2| - |\Gamma_1| \le \frac{\alpha}{4}$.

To determine the lower bound, we apply the condition $J^2 \le 2|L|(M + 2|L|)$, which simplifies to $64 \le 8(12\alpha + 8)$, or $12\alpha \ge 0$. This inequality holds for all $\alpha \in (0,1]$, and the corresponding branch of Lemma \ref{L6} gives $-\Phi(p_1, p_2) \le 2J \sqrt{\frac{2|L|}{M + 2|L|}} = 16\sqrt{\frac{2}{3\alpha + 2}}$. Substituting this into \eqref{ND16} yields the lower bound.
\subsection*{Sharpness of the Estimates and Extremal Functions}
\begin{itemize}
\item The upper bound is reached for the odd starlike function $f_{\text{odd}} \in \mathcal{S}_{ex}^{\ast}$ defined by:
\[
f_{\text{odd}}(z) = z\exp\left(\int_0^z \frac{e^{\alpha t^2}-1}{t}dt\right),
\]
which corresponds to $p_1=0$ and $p_2=2$. 

\item The lower bound is attained by the function $f_{\text{spiral}} \in \mathcal{S}_{ex}^{\ast}$ defined by:
\[
f_{\text{spiral}}(z) = z\exp\left(\int_0^z \frac{e^{\alpha \omega(t)}-1}{t}dt\right), \quad \text{where} \quad \omega(z)=\frac{z\left(\sqrt{2}+\sqrt{3\alpha+2}z\right)}{\sqrt{3\alpha+2}+\sqrt{2}z},
\]
which corresponds to $p_1=2\sqrt{\frac{2}{3\alpha+2}}$ and $p_2=2$.
\end{itemize}
\end{proof}

\section{Hankel Determinants for the Inverse Logarithmic Coefficients}

\begin{theo}\label{T3}
Let $0<\alpha\le 1$. If $f\in \mathcal{S}_{ex}^{\ast}$ and the inverse logarithmic coefficients $\Gamma_n$ $(n=1,2,3)$ are defined by \eqref{IG1}, then the second-order inverse logarithmic Hankel determinant $H_{2,1}(F_{f^{-1}}/2)$ satisfies the sharp inequalities:
\begin{equation}\label{ND18}
\left| H_{2,1}\left(F_{f^{-1}}/2\right)\right| \le
\begin{cases}
\dfrac{\alpha^2}{16}, & 0<\alpha\le \dfrac25,\\[8pt]
\dfrac{\alpha^2(15\alpha^2+10\alpha+4)}{4(35\alpha^2+60\alpha+12)}, & \dfrac25<\alpha\le 1.
\end{cases}
\end{equation}
These inequalities are sharp for all $\alpha\in(0,1]$.
\end{theo}

\begin{proof}
For $f\in\mathcal{S}_{ex}^{\ast}$, the direct coefficients are expressed via $p\in\mathcal{P}$, and substituting them into
\[
H_{2,1}(F_{f^{-1}}/2)=\Gamma_1\Gamma_3-\Gamma_2^2
\]
yields
\begin{equation}\label{ND20}
\begin{aligned}
H_{2,1}\!\left(F_{f^{-1}}/2\right)
&=
\frac{\alpha^2(35\alpha^2+60\alpha+12)}{9216}p_1^4
-\frac{\alpha^2(5\alpha+2)}{384}p_1^2p_2\\
&\quad
+\frac{\alpha^2}{48}p_1p_3
-\frac{\alpha^2}{64}p_2^2.
\end{aligned}
\end{equation}

With $p_1=2x$, $x\in[0,1]$, and the parametrization of Lemma~\ref{L1} for $p_2,p_3$, we obtain
\begin{equation}\label{ND21}
\begin{aligned}
H_{2,1}\!\left(F_{f^{-1}}/2\right)
&=
\frac{35\alpha^4}{576}x^4
-\frac{5\alpha^3}{48}x^2(1-x^2)\tau_2
-\frac{\alpha^2}{48}(1-x^2)(x^2+3)\tau_2^2\\
&\quad
+\frac{\alpha^2}{12}x(1-x^2)(1-|\tau_2|^2)\tau_3.
\end{aligned}
\end{equation}

Using $|\tau_3|\le 1$ and the triangle inequality,
\begin{equation}\label{ND22}
\bigl|H_{2,1}(F_{f^{-1}}/2)\bigr|
\le
\frac{\alpha^2 x(1-x^2)}{12}
\Bigl(
|A+B\tau_2+C\tau_2^2|+1-|\tau_2|^2
\Bigr),
\end{equation}
where
\[
A=\frac{35\alpha^2 x^3}{48(1-x^2)}>0,
\qquad
B=-\frac{5\alpha}{4}x,
\qquad
C=-\frac{3+x^2}{4x}<0.
\]
Since $AC<0$, we apply Lemma~\ref{L2}(ii).

\subsection*{Subcase analysis}

For $0<x\le 1$,
\[
|C|=\frac{3+x^2}{4x}\ge 1.
\]
Hence $2(1-|C|)\le 0$, while $|B|>0$. Therefore
\[
|B|<2(1-|C|)
\]
is impossible, so Subcase (a) is empty.

Subcase (b) requires
\[
B^2<\min\left\{4(1+|C|)^2,\,-4AC(C^{-2}-1)\right\}.
\]
Since $|C|\ge 1$, we have $C^{-2}-1\le 0$, and because $A>0$ and $C<0$, it follows that
\[
-4AC(C^{-2}-1)\le 0.
\]
But $B^2>0$, so the second condition cannot hold. Thus Subcase (b) is empty.

\paragraph{Subcase (c).}
Subcase (c) requires
\[
|B|\ge 2(1-|C|)
\]
and
\[
|C|(|B|+4|A|)\le |AB|.
\]
The first condition holds automatically. The second condition, after substituting $A,B,C$ and writing $y=x^2$, becomes
\[
180+(420\alpha-120)y+(140\alpha-60-175\alpha^2)y^2\le 0.
\]
The leading coefficient is negative because
\[
140\alpha-60-175\alpha^2<0
\]
for all $\alpha>0$. At $y=0$, the left-hand side equals $180>0$, and at $y=1$ it equals
\[
35\alpha(16-5\alpha)>0
\]
for $0<\alpha\le 1$. Since the quadratic is concave and positive at both endpoints of $[0,1]$, it is positive on the whole interval. Hence Subcase (c) is empty.

\paragraph{Subcase (d).}
Subcase (d) is governed by
\[
|AB|\le |C|(|B|-4|A|).
\]
With $y=x^2$, this is equivalent to
\begin{equation}\label{Qineq}
Q(y):=(175\alpha^2+140\alpha+60)y^2+(420\alpha+120)y-180\le 0.
\end{equation}
Since $Q(0)=-180<0$ and the leading coefficient is positive, $Q(y)\le 0$ exactly for
\[
0\le y\le y_d,
\]
where $y_d$ is the unique positive root of $Q(y)=0$. One checks that $y_d<1$ for all $\alpha\in(0,1]$.

In this range, Lemma~\ref{L2} gives
\[
|A+B\tau_2+C\tau_2^2|+1-|\tau_2|^2
=
-|A|+|B|+|C|.
\]
Substituting this into \eqref{ND22} and simplifying yields
\begin{equation}\label{phi_d}
\Phi_d(x)
=
\frac{\alpha^2}{576}
\Bigl[
36+12(5\alpha-2)x^2-(35\alpha^2+60\alpha+12)x^4
\Bigr],
\qquad x\in[0,\sqrt{y_d}].
\end{equation}

\paragraph{Maximizing $\Phi_d$ on $[0,\sqrt{y_d}]$.}
Write $y=x^2$. Then
\[
\Phi_d(x)
=
\frac{\alpha^2}{576}
\left[
36+12(5\alpha-2)y-(35\alpha^2+60\alpha+12)y^2
\right].
\]
This is a concave quadratic in $y$. Its vertex is at
\[
y_v=\frac{6(5\alpha-2)}{35\alpha^2+60\alpha+12}.
\]

If $0<\alpha\le \frac25$, then $y_v\le 0$. Hence $\Phi_d$ is decreasing on $[0,y_d]$, and
\[
\max_{[0,y_d]}\Phi_d=\Phi_d(0)=\frac{\alpha^2}{16}.
\]

If $\frac25<\alpha\le 1$, then $y_v>0$. A direct calculation gives $Q(y_v)<0$, so $y_v<y_d$. Hence
\[
\max_{[0,y_d]}\Phi_d
=
\Phi_d(\sqrt{y_v})
=
\frac{\alpha^2}{576}
\left(
36+\frac{36(5\alpha-2)^2}{35\alpha^2+60\alpha+12}
\right)
=
\frac{\alpha^2(15\alpha^2+10\alpha+4)}
{4(35\alpha^2+60\alpha+12)}.
\]

\paragraph{Subcase (e): the region $y\in(y_d,1)$.}
For $y>y_d$, Lemma~\ref{L2} gives
\[
|A+B\tau_2+C\tau_2^2|+1-|\tau_2|^2
=
(|C|+|A|)\sqrt{1-\frac{B^2}{4AC}}.
\]
A direct computation gives
\[
|C|+|A|
=
\frac{36-24x^2+(35\alpha^2-12)x^4}{48x(1-x^2)},
\]
and
\[
1-\frac{B^2}{4AC}
=
\frac{36-8x^2}{7(3+x^2)}.
\]
Therefore
\begin{equation}\label{phi_e_correct}
\Phi_e(x)
=
\frac{\alpha^2}{576}
\left[
36-24x^2+(35\alpha^2-12)x^4
\right]
\sqrt{\frac{36-8x^2}{7(3+x^2)}},
\qquad x\in(\sqrt{y_d},1).
\end{equation}

Set $y=x^2$ and consider
\[
F(y)=[\Phi_e(\sqrt{y})]^2
=
\frac{\alpha^4}{576^2\cdot 7}
\left[
(35\alpha^2-12)y^2-24y+36
\right]^2
\frac{36-8y}{3+y},
\qquad y\in(y_d,1).
\]

Differentiating,
\[
F'(y)
=
\frac{\alpha^4}{576^2\cdot 7}
\cdot
\frac{P(y)}{(3+y)^2}
\cdot
\widetilde H(y),
\]
where
\[
P(y)=(35\alpha^2-12)y^2-24y+36,
\]
and
\begin{equation}\label{Htilde}
\widetilde H(y)
=
(-1120\alpha^2+384)y^3
+
(-420\alpha^2+528)y^2
+
(15120\alpha^2-4320)y
-
7344.
\end{equation}

Now we need the elementary sign lemma \eqref{lem:Hsign}.

We also note that
\[
P(y)=(35d-12)y^2-24y+36>0
\]
for $0\le y\le1$. Indeed, if $d\le \frac{12}{35}$, then the coefficient is non-positive, the parabola is concave, and since $P(0)=36>0$ and $P(1)=35d>0$, we have $P(y)>0$. If $d>\frac{12}{35}$, then the vertex satisfies
\[
y_0=\frac{12}{35d-12}.
\]
If $d\le \frac{24}{35}$, then $y_0\ge1$, so $P$ is decreasing on $[0,1]$, and $P(1)=35d>0$ gives $P(y)>0$. If $d>\frac{24}{35}$, then
\[
P(y_0)=36-\frac{144}{35d-12}>24>0.
\]
Hence $P(y)>0$ on $[0,1]$.

Therefore the sign of $F'(y)$ is the same as the sign of $\widetilde H(y)$ on $(y_d,1)$. By Lemma~\ref{lem:Hsign}, $\widetilde H$ cannot change sign from positive to negative on this interval. Consequently, $F$ has no interior local maximum on $(y_d,1)$. Hence the maximum of $\Phi_e$ on $[\sqrt{y_d},1)$ occurs at one of the endpoints.

At the left endpoint,
\[
\Phi_e(\sqrt{y_d})=\Phi_d(\sqrt{y_d})
\le
\max_{[0,\sqrt{y_d}]}\Phi_d.
\]

At the right endpoint $y=1$, we have $x=1$. Substituting $x=1$ into \eqref{ND21}, all terms containing $(1-x^2)$ vanish, and we obtain
\[
H_{2,1}\!\left(F_{f^{-1}}/2\right)
=
\frac{35\alpha^4}{576}.
\]
Thus
\begin{equation}\label{boundary_hankel_correct}
|H_{2,1}(F_{f^{-1}}/2)|_{x=1}
=
\frac{35\alpha^4}{576}.
\end{equation}

We now compare the candidate maxima.

\paragraph{Left boundary: $x=0$.}
At $x=0$,
\[
|H_{2,1}(F_{f^{-1}}/2)|_{x=0}
=
\frac{\alpha^2}{16}
=
\frac{36\alpha^2}{576}.
\]

\paragraph{Right boundary: $x=1$.}
At $x=1$,
\[
|H_{2,1}(F_{f^{-1}}/2)|_{x=1}
=
\frac{35\alpha^4}{576}.
\]
Since $\alpha\le 1$, we have $35\alpha^2\le 35<36$, so
\[
\frac{35\alpha^4}{576}
<
\frac{36\alpha^2}{576}
=
\frac{\alpha^2}{16}.
\]
Therefore the right boundary never dominates the left boundary.

\paragraph{Subcase (d) vertex.}
For $\frac25<\alpha\le 1$, the Subcase (d) maximum is
\[
C_d(\alpha)
=
\frac{\alpha^2(15\alpha^2+10\alpha+4)}{4(35\alpha^2+60\alpha+12)}.
\]
Comparing with $\alpha^2/16$,
\[
C_d(\alpha)-\frac{\alpha^2}{16}
=
\frac{\alpha^2(5\alpha-2)^2}
{16(35\alpha^2+60\alpha+12)}
\ge0.
\]
Thus
\[
C_d(\alpha)\ge \frac{\alpha^2}{16}
\]
for all $\alpha\in(\frac25,1]$, with equality precisely at $\alpha=\frac25$.

Therefore the global maximum is
\[
M_H(\alpha)
=
\begin{cases}
\dfrac{\alpha^2}{16}, & 0<\alpha\le \dfrac25,\\[8pt]
\dfrac{\alpha^2(15\alpha^2+10\alpha+4)}{4(35\alpha^2+60\alpha+12)}, & \dfrac25<\alpha\le 1.
\end{cases}
\]

This proves \eqref{ND18}.

\subsection*{Sharpness}

For $0<\alpha\le \frac25$, choose $x=0$, i.e.\ $p_1=0$. Then from \eqref{ND21},
\[
H_{2,1}(F_{f^{-1}}/2)
=
-\frac{\alpha^2}{48}\cdot 3\tau_2^2
=
-\frac{\alpha^2}{16}\tau_2^2.
\]
Taking $\tau_2=1$ gives
\[
|H_{2,1}(F_{f^{-1}}/2)|=\frac{\alpha^2}{16}.
\]
The corresponding extremal function is the odd starlike function
\[
f_{\rm odd}(z)
=
z\exp\left(
\int_0^z \frac{e^{\alpha t^2}-1}{t}\,dt
\right).
\]

For $\frac25<\alpha\le 1$, choose
\[
\tau_1=\sqrt{y_v},
\qquad
\tau_2=1,
\qquad
\tau_3=-1
\]
(the value of $\tau_3$ is immaterial because $|\tau_2|=1$). With $A>0$, $B<0$, $C<0$, and the Subcase (d) condition $A\le |B|+|C|$, we have
\[
|A+B+C|
=
|A-|B|-|C||
=
|B|+|C|-A
=
-A+|B|+|C|.
\]
Thus the required equality in Lemma~\ref{L2} is achieved.

Lemma~\ref{L1} guarantees a $p\in\mathcal{P}$ with these parameters, and the associated $f_{\rm Hankel}\in\mathcal{S}_{ex}^{\ast}$ attains the bound
\[
\frac{\alpha^2(15\alpha^2+10\alpha+4)}{4(35\alpha^2+60\alpha+12)}.
\]

This completes the proof.
\end{proof}
\section{Hermitian--Toeplitz Determinants for the Class $\mathcal{S}_{ex}^{\ast}$}

For a sequence $\{a_k\}_{k=2}^{\infty}$ of coefficients of a normalized analytic function $f\in\mathcal{A}$ of the form \eqref{eq1}, the Hermitian--Toeplitz determinant of order $q$ starting at index $n$ is defined by
\begin{equation}\label{ND24}
T_{q,n}(f) :=
\begin{vmatrix}
a_n & a_{n+1} & \cdots & a_{n+q-1} \\
\overline{a_{n+1}} & a_n & \cdots & a_{n+q-2} \\
\vdots & \vdots & \ddots & \vdots \\
\overline{a_{n+q-1}} & \overline{a_{n+q-2}} & \cdots & a_n
\end{vmatrix}.
\end{equation}
Setting $q=3$ and $n=1$ in \eqref{ND24} and using $a_1=1$, the third-order Hermitian--Toeplitz determinant reduces to
\begin{equation}\label{ND25}
T_{3,1}(f) = 2\Re\left(a_{2}^{2}\overline{a_{3}}\right) - 2|a_{2}|^{2} - |a_{3}|^{2} + 1.
\end{equation}

\begin{theo}\label{T4}
Let $0<\alpha\le 1$. If $f\in\mathcal{S}_{ex}^{\ast}$ is given by \eqref{eq1}, then the third-order Hermitian--Toeplitz determinant satisfies the sharp inequalities
\begin{equation}\label{ND26}
\mathcal{B}_L(\alpha)\le T_{3,1}(f)\le 1,
\end{equation}
where the sharp lower bound $\mathcal{B}_L(\alpha)$ is defined by
\begin{equation}\label{ND27}
\mathcal{B}_L(\alpha)=
\begin{cases}
1-2\alpha^2+\dfrac{15\alpha^4}{16}, & 0<\alpha<\alpha_H,\\[6pt]
1-\dfrac{\alpha^2}{4}-\dfrac{\alpha^2(\alpha+6)^2}{4(15\alpha^2+4\alpha-4)}, & \alpha_H\le\alpha\le1,
\end{cases}
\end{equation}
with
\[
\alpha_H:=\frac{-2+\sqrt{964}}{30}
=\frac{-1+\sqrt{241}}{15}
\approx0.9683.
\]
Both bounds are sharp for all $\alpha\in(0,1]$.
\end{theo}

\begin{proof}
For $f\in\mathcal{S}_{ex}^{\ast}$, the coefficient expansions are
\begin{equation}\label{ND28}
a_2=\frac{\alpha p_1}{2},
\qquad
a_3=\frac{\alpha p_2}{4}+\frac{3\alpha^2-2\alpha}{16}p_1^2.
\end{equation}
Using the Libera--Z{\l}otkiewicz parametrization
\[
2p_2=p_1^2+(4-p_1^2)\xi,\qquad \xi\in\overline{\mathbb{D}},
\]
we may write
\[
a_3=\frac{3\alpha^2}{16}p_1^2+\frac{\alpha}{8}(4-p_1^2)\xi.
\]
By rotational invariance, assume $p_1\in[0,2]$ and set $x:=p_1^2\in[0,4]$. Substituting the expressions for $a_2$ and $a_3$ into \eqref{ND25} gives
\begin{equation}\label{T31expr}
T_{3,1}(f)
=
\frac{1}{256}\Bigl(
256-128\alpha^2 x+15\alpha^4 x^2
+4\alpha^3 x(4-x)\Re\xi
-4\alpha^2(4-x)^2|\xi|^2
\Bigr).
\end{equation}

\subsection*{Upper bound}
For fixed $x$, the coefficient of $\Re\xi$ is $4\alpha^3 x(4-x)\ge0$. Hence the numerator in \eqref{T31expr} is maximized when $\Re\xi=|\xi|=:y$ with $y\in[0,1]$. Define
\[
F(x,y)
=
256-128\alpha^2 x+15\alpha^4 x^2
+4\alpha^3 x(4-x)y
-4\alpha^2(4-x)^2 y^2.
\]
Then $T_{3,1}(f)\le\frac1{256}F(x,y)$, and we maximize $F$ on $[0,4]\times[0,1]$.

The partial derivatives are
\begin{align*}
\frac{\partial F}{\partial y}
&=
4\alpha^2(4-x)\bigl(\alpha x-2(4-x)y\bigr),\\
\frac{\partial F}{\partial x}
&=
-128\alpha^2+30\alpha^4 x
+4\alpha^3(4-2x)y
+8\alpha^2(4-x)y^2.
\end{align*}
Setting $\frac{\partial F}{\partial y}=0$ gives
\[
y_c=\frac{\alpha x}{2(4-x)}
\]
for $x\in(0,4)$. Substituting this into $\frac{\partial F}{\partial x}=0$ yields $x_c=4/\alpha^2$. Since $0<\alpha\le1$, we have $x_c\ge4$, so $F$ has no interior critical point. Thus the maximum occurs on the boundary.

On the left boundary $x=0$,
\[
F(0,y)=256-64\alpha^2 y^2\le256,
\]
with equality at $y=0$. On the right boundary $x=4$,
\[
F(4,y)=256-512\alpha^2+240\alpha^4<256.
\]
On the lower boundary $y=0$,
\[
F(x,0)=256-128\alpha^2 x+15\alpha^4 x^2,
\]
and
\[
\frac{d}{dx}F(x,0)
=
-128\alpha^2+30\alpha^4 x
\le
-128\alpha^2+120\alpha^4
=
-8\alpha^2(16-15\alpha^2)<0,
\]
so $F(x,0)$ is decreasing and its maximum is $F(0,0)=256$.

On the upper boundary $y=1$,
\[
F(x,1)
=
256-64\alpha^2
+(-96\alpha^2+16\alpha^3)x
+(15\alpha^4-4\alpha^3-4\alpha^2)x^2.
\]
Set $h(x)=F(x,1)$. If $0<\alpha\le\frac23$, then $h$ is concave with vertex at a negative $x$-coordinate, so $h$ is decreasing on $[0,4]$, and
\[
\max_{[0,4]}h=h(0)=256-64\alpha^2<256.
\]
If $\frac23<\alpha\le1$, then $h$ is convex, so its maximum on $[0,4]$ is attained at an endpoint. Since
\[
h(0)=256-64\alpha^2<256,
\qquad
h(4)=256-512\alpha^2+240\alpha^4<256,
\]
we again have $h(x)<256$ for all $x\in[0,4]$.

Therefore the global maximum of $F$ is $256$, attained at $(x,y)=(0,0)$. Dividing by $256$ gives $T_{3,1}(f)\le1$.

Equality is attained by the $3$-fold symmetric function
\[
f_{\rm tri}(z)
=
z\exp\left(\int_0^z\frac{e^{\alpha t^3}-1}{t}\,dt\right),
\]
for which $p_1=p_2=0$, hence $a_2=a_3=0$ and $T_{3,1}(f_{\rm tri})=1$.

\subsection*{Lower bound}
To minimize $T_{3,1}(f)$, we use the opposite extremal choice $\Re\xi=-|\xi|=-y$, $y\in[0,1]$. Then the numerator becomes
\[
N(x,y)
=
256-128\alpha^2 x+15\alpha^4 x^2
-4\alpha^3 x(4-x)y
-4\alpha^2(4-x)^2 y^2.
\]
For fixed $x\in[0,4)$,
\[
\frac{\partial N}{\partial y}
=
-4\alpha^3 x(4-x)-8\alpha^2(4-x)^2 y<0,
\]
so the minimum in $y$ occurs at $y=1$. Thus the lower bound is obtained from
\begin{equation}\label{ND35}
\phi(x):=N(x,1)
=
(15\alpha^4+4\alpha^3-4\alpha^2)x^2
-(16\alpha^3+96\alpha^2)x
+256-64\alpha^2,
\qquad x\in[0,4].
\end{equation}
Its derivative is
\[
\phi'(x)
=
2(15\alpha^4+4\alpha^3-4\alpha^2)x
-(16\alpha^3+96\alpha^2).
\]
At $x=4$,
\[
\phi'(4)=8\alpha^2(15\alpha^2+2\alpha-16).
\]
The quadratic $15\alpha^2+2\alpha-16$ has the unique positive root
\[
\alpha_H
=
\frac{-2+\sqrt{964}}{30}
=
\frac{-1+\sqrt{241}}{15}
\approx0.9683.
\]

If $0<\alpha<\alpha_H$, then $\phi'(4)<0$. Since $\phi'$ is linear, this implies $\phi'(x)<0$ for all $x\in[0,4]$. Hence $\phi$ is decreasing, and
\[
\min_{[0,4]}\phi=\phi(4)
=
256-512\alpha^2+240\alpha^4.
\]
Dividing by $256$ gives
\[
\mathcal{B}_L(\alpha)
=
1-2\alpha^2+\frac{15\alpha^4}{16}.
\]

If $\alpha_H\le\alpha\le1$, then $\phi'(4)\ge0$. Since the leading coefficient $15\alpha^4+4\alpha^3-4\alpha^2$ is positive, $\phi$ is convex, and $\phi'$ has exactly one zero in $(0,4]$, namely
\[
x_{\rm v}
=
\frac{8(\alpha+6)}{15\alpha^2+4\alpha-4}.
\]
Thus the minimum occurs at $x_{\rm v}$. Evaluating $\phi(x_{\rm v})$ gives
\[
\phi(x_{\rm v})
=
256-64\alpha^2-\frac{64\alpha^2(\alpha+6)^2}{15\alpha^2+4\alpha-4},
\]
so
\[
\mathcal{B}_L(\alpha)
=
1-\frac{\alpha^2}{4}
-\frac{\alpha^2(\alpha+6)^2}{4(15\alpha^2+4\alpha-4)}.
\]
For $\alpha=1$, this reduces to $-\frac1{15}$, in agreement with the known bound for $\mathcal{S}_e^*$.

\subsection*{Sharpness of the lower bound}
The lower bound is attained when $y=|\xi|=1$ and $\Re\xi=-1$.

For $0<\alpha<\alpha_H$, choose $p_1=2$, equivalently $x=4$, and $\xi=-1$. This gives the Koebe-type function
\[
f_{\rm Koebe}(z)
=
z\exp\left(\int_0^z\frac{e^{\alpha t}-1}{t}\,dt\right),
\]
for which
\[
T_{3,1}(f_{\rm Koebe})
=
1-2\alpha^2+\frac{15\alpha^4}{16}.
\]

For $\alpha_H\le\alpha\le1$, choose
\[
p_1=\sqrt{x_{\rm v}},
\qquad
\xi=-1.
\]
The Libera--Z{\l}otkiewicz theorem guarantees the existence of a Carath\'eodory function with these parameters. The associated function $f\in\mathcal{S}_{ex}^{\ast}$ then attains the stated lower bound.

This completes the proof.
\end{proof}

\section{Generalized Fekete--Szeg\H{o} Functional for the Class $\mathcal{S}_{ex}^{\ast}$}

Lecko and Partyka \cite{Lecko2024} introduced the generalized Fekete--Szeg\H{o} functional
\begin{equation}\label{FG}
F_{\lambda,\mu}(f)
=
\bigl|a_3-\lambda a_2^2\bigr|
-
\mu |a_2|,
\qquad
\lambda\in\mathbb{C},\quad \mu>0,
\end{equation}
for functions in the univalent class $\mathcal{S}$. Later, Bulboac\u{a} et al.\ \cite{Bulboaca2025} studied this functional for the class $\mathcal{S}$ and its convex subclass. In this section, we give a complete solution for the exponential Ma--Minda starlike class $\mathcal{S}_{ex}^{\ast}$.

\begin{theo}\label{T5}
Let $0<\alpha\le1$, $\lambda\in\mathbb{C}$, and $\mu>0$. If $f\in\mathcal{S}_{ex}^{\ast}$ is given by \eqref{eq1}, then
\begin{equation}\label{FGbounds}
\mathcal{B}_L
\le
|a_3-\lambda a_2^2|-\mu|a_2|
\le
\mathcal{B}_U,
\end{equation}
where
\begin{equation}\label{FGU}
\mathcal{B}_U
=
\begin{cases}
\dfrac{\alpha}{2}, & \alpha|3-4\lambda| < 2+4\mu,\\[6pt]
\dfrac{\alpha^2|3-4\lambda|-4\alpha\mu}{4}, & \alpha|3-4\lambda| \ge 2+4\mu,
\end{cases}
\end{equation}
and
\begin{equation}\label{FGL}
\mathcal{B}_L
=
\begin{cases}
-\dfrac{\alpha}{4}\bigl(4\mu-\alpha|3-4\lambda|\bigr), & \alpha|3-4\lambda| \le 2\mu-2,\\[8pt]
-\alpha\mu\sqrt{\dfrac{2}{\alpha|3-4\lambda|+2}}, & \alpha|3-4\lambda| \ge 2\mu^2-2,\\[8pt]
-\dfrac{\alpha}{2}-\dfrac{\alpha\mu^2}{\alpha|3-4\lambda|+2}, & 2\mu-2<\alpha|3-4\lambda|<2\mu^2-2,
\end{cases}
\end{equation}
whenever the corresponding region is nonempty. All bounds are sharp.
\end{theo}

\begin{proof}
For $f\in\mathcal{S}_{ex}^{\ast}$, write
\[
\frac{zf'(z)}{f(z)}
=
e^{\alpha\omega(z)},
\qquad
\omega
=
\frac{p-1}{p+1},
\quad
p\in\mathcal{P}.
\]
Comparing coefficients gives
\begin{equation}\label{FGa2a3}
a_2=\frac{\alpha p_1}{2},
\qquad
a_3
=
\frac{\alpha p_2}{4}
+
\frac{3\alpha^2-2\alpha}{16}p_1^2.
\end{equation}
Substituting \eqref{FGa2a3} into the functional yields
\begin{equation}\label{FGPhi}
\begin{aligned}
|a_3-\lambda a_2^2|-\mu|a_2|
&=
\left|
\frac{\alpha p_2}{4}
+
\frac{3\alpha^2-2\alpha-4\lambda\alpha^2}{16}p_1^2
\right|
-\frac{\alpha\mu}{2}|p_1|\\
&=
\frac{\alpha}{4}
\left(
\left|
p_2+\frac{\alpha(3-4\lambda)-2}{4}p_1^2
\right|
-2\mu|p_1|
\right)\\
&=
\frac{\alpha}{4}\Phi(p_1,p_2),
\end{aligned}
\end{equation}
where
\[
\Phi(p_1,p_2)
=
|Lp_2+Kp_1^2|-J|p_1|,
\]
with
\[
K=\frac{\alpha(3-4\lambda)-2}{4},
\qquad
L=1,
\qquad
J=2\mu.
\]
Set
\[
D=\alpha|3-4\lambda|,
\qquad
\eta=\frac{3-4\lambda}{|3-4\lambda|}
\quad
\text{if }D>0,
\]
and take $\eta=1$ if $D=0$. Then
\[
\alpha(3-4\lambda)=D\eta.
\]
The quantities needed in Lemma~\ref{L6} are
\[
|2K+L|=\frac{D}{2},
\qquad
M=|4K+2L|=D.
\]

\subsection*{Upper bound}
By Lemma~\ref{L6}, if
\[
|2K+L|<|L|+J,
\]
then
\[
D<2+4\mu,
\]
and
\[
\Phi\le 2|L|=2.
\]
Hence
\[
|a_3-\lambda a_2^2|-\mu|a_2|
\le
\frac{\alpha}{2}.
\]

If
\[
|2K+L|\ge |L|+J,
\]
then
\[
D\ge2+4\mu,
\]
and
\[
\Phi\le |4K+2L|-2J=D-4\mu.
\]
Thus
\[
|a_3-\lambda a_2^2|-\mu|a_2|
\le
\frac{\alpha^2|3-4\lambda|-4\alpha\mu}{4}.
\]
This proves \eqref{FGU}.

\subsection*{Lower bound}
For the lower bound, we use the estimates for $-\Phi$ from Lemma~\ref{L6}.

If
\[
J\ge M+2|L|,
\]
then
\[
2\mu\ge D+2,
\]
equivalently
\[
D\le2\mu-2.
\]
In this case,
\[
-\Phi\le 2J-M=4\mu-D,
\]
and hence
\[
|a_3-\lambda a_2^2|-\mu|a_2|
\ge
-\frac{\alpha}{4}\bigl(4\mu-\alpha|3-4\lambda|\bigr).
\]

If
\[
J^2\le2|L|(M+2|L|),
\]
then
\[
4\mu^2\le2(D+2),
\]
equivalently
\[
D\ge2\mu^2-2.
\]
Then
\[
-\Phi
\le
2J\sqrt{\frac{2|L|}{M+2|L|}}
=
4\mu\sqrt{\frac{2}{D+2}},
\]
so
\[
|a_3-\lambda a_2^2|-\mu|a_2|
\ge
-\alpha\mu\sqrt{\frac{2}{D+2}}.
\]

If
\[
2\mu-2<D<2\mu^2-2,
\]
which can occur only when $\mu>1$, then Lemma~\ref{L6} gives
\[
-\Phi
\le
2|L|+\frac{J^2}{M+2|L|}
=
2+\frac{4\mu^2}{D+2}.
\]
Therefore
\[
|a_3-\lambda a_2^2|-\mu|a_2|
\ge
-\frac{\alpha}{2}-\frac{\alpha\mu^2}{D+2}.
\]
These three cases give \eqref{FGL}.

\subsection*{Sharpness}
We exhibit Carath\'eodory parameters attaining equality in each branch.

\paragraph{Upper bound, first branch: $D<2+4\mu$.}
Choose
\[
p_1=0,\qquad p_2=2.
\]
Then $a_2=0$, $a_3=\alpha/2$, and
\[
|a_3-\lambda a_2^2|-\mu|a_2|
=
\frac{\alpha}{2}.
\]
The extremal function is
\[
f_{\mathrm{odd}}(z)
=
z\exp\left(\int_0^z\frac{e^{\alpha t^2}-1}{t}\,dt\right).
\]

\paragraph{Upper bound, second branch: $D\ge2+4\mu$.}
Choose
\[
p_1=2,\qquad p_2=2.
\]
Then
\[
a_2=\alpha,
\qquad
a_3=\frac{3\alpha^2}{4},
\]
and
\[
|a_3-\lambda a_2^2|-\mu|a_2|
=
\frac{\alpha^2|3-4\lambda|}{4}-\alpha\mu.
\]
The extremal function is the Koebe-type function
\[
f_{\mathrm{Koebe}}(z)
=
z\exp\left(\int_0^z\frac{e^{\alpha t}-1}{t}\,dt\right).
\]

\paragraph{Lower bound, first branch: $D\le2\mu-2$.}
Again choose $p_1=p_2=2$. Then
\[
|a_3-\lambda a_2^2|-\mu|a_2|
=
\frac{\alpha D}{4}-\alpha\mu
=
-\frac{\alpha}{4}(4\mu-D).
\]

\paragraph{Lower bound, second branch: $D\ge2\mu^2-2$.}
Set
\[
p_1=2\sqrt{\frac{2}{D+2}},
\qquad
\tau_2=-\eta.
\]
By Lemma~\ref{L1}, this choice is admissible. The corresponding second coefficient is
\[
p_2
=
\frac{4-2D\eta}{D+2}.
\]
A direct calculation gives
\[
Kp_1^2+p_2=0,
\]
so
\[
\Phi(p_1,p_2)
=
-J|p_1|
=
-4\mu\sqrt{\frac{2}{D+2}}.
\]
Therefore
\[
|a_3-\lambda a_2^2|-\mu|a_2|
=
-\alpha\mu\sqrt{\frac{2}{D+2}}.
\]

\paragraph{Lower bound, third branch: $2\mu-2<D<2\mu^2-2$.}
Set
\[
x=\frac{2\mu}{D+2},
\qquad
p_1=2x,
\qquad
\tau_2=-\eta.
\]
Since $D>2\mu-2$, we have $x<1$, so this choice is admissible. The second coefficient is
\[
p_2=2x^2-2\eta(1-x^2),
\]
and a computation gives
\[
Kp_1^2+p_2
=
\eta\left(\frac{4\mu^2}{D+2}-2\right).
\]
Because $D<2\mu^2-2$,
\[
|Kp_1^2+p_2|
=
\frac{4\mu^2}{D+2}-2.
\]
Also,
\[
J|p_1|
=
\frac{8\mu^2}{D+2}.
\]
Thus
\[
\Phi(p_1,p_2)
=
-2-\frac{4\mu^2}{D+2},
\]
and hence
\[
|a_3-\lambda a_2^2|-\mu|a_2|
=
-\frac{\alpha}{2}-\frac{\alpha\mu^2}{D+2}.
\]

This completes the proof.
\end{proof}
\section*{Conclusion}

We have obtained sharp coefficient estimates for the parametric exponential Ma--Minda starlike class
\[
\mathcal{S}_{ex}^{\ast}
=
\left\{
f\in\mathcal{S}:
\frac{zf'(z)}{f(z)}
\prec e^{\alpha z},\ 0<\alpha\le1
\right\}.
\]
The main results are summarized in Table~\ref{tab:summary}. For the inverse logarithmic coefficients, we proved the sharp bounds
\[
|\Gamma_1|\le \frac{\alpha}{2},
\qquad
|\Gamma_2|\le
\begin{cases}
\dfrac{\alpha}{4}, & 0<\alpha\le\dfrac23,\\[4pt]
\dfrac{3\alpha^2}{8}, & \dfrac23<\alpha\le1,
\end{cases}
\]
and established a four-branch structure for \(|\Gamma_3|\), with transition points
\[
\alpha_\star\approx0.542712,\qquad
\alpha_b\approx0.679103,\qquad
\alpha_c\approx0.691095.
\]
Sharp bounds were also obtained for the difference \(|\Gamma_2|-|\Gamma_1|\), the inverse logarithmic Hankel determinant \(H_{2,1}(F_{f^{-1}}/2)\), and the Hermitian--Toeplitz determinant \(T_{3,1}(f)\), whose lower bound changes at
\[
\alpha_H=\frac{-1+\sqrt{241}}{15}\approx0.9683.
\]
Finally, we completely solved the generalized Fekete--Szeg\H{o} problem for \(\mathcal{S}_{ex}^{\ast}\).

All bounds are sharp, and the corresponding extremal functions are explicitly constructed from the integral representation
\[
f(z)
=
z\exp\left(
\int_0^z \frac{1}{t}
\left[
\exp\left(
\alpha\frac{p(t)-1}{p(t)+1}
\right)
-1
\right]
dt
\right).
\]
For \(\alpha=1\), our results reduce to and extend several known estimates for the exponential starlike class \(\mathcal{S}_e^*\). The method is based on the Carath\'eodory parametrization and the optimization lemmas of Choi--Kim--Sugawa and Sim--Thomas, and is likely to apply to other Ma--Minda subclasses as well.

\begin{table}[htbp]
\centering
\caption{Summary of sharp bounds and extremal functions for \(\mathcal{S}_{ex}^{\ast}\).}
\label{tab:summary}
\setlength{\tabcolsep}{3pt}
\scriptsize
\begin{tabular}{|p{1.8cm}|p{3.2cm}|p{3.8cm}|p{4.4cm}|}
\hline
\textbf{Result} & \textbf{Parameter range} & \textbf{Sharp bound} & \textbf{Extremal function/parameters} \\
\hline
\(|\Gamma_1|\) & \(0<\alpha\le1\) & \(\displaystyle\frac{\alpha}{2}\) & \(f_{\rm Koebe}\), \(p_1=2\) \\
\hline
\multirow{2}{*}{\(|\Gamma_2|\)}
& \(0<\alpha\le\frac23\) & \(\displaystyle\frac{\alpha}{4}\) & \(f_{\rm odd}\), \(p_1=0,\ p_2=2\) \\
\cline{2-4}
& \(\frac23<\alpha\le1\) & \(\displaystyle\frac{3\alpha^2}{8}\) & \(f_{\rm Koebe}\), \(p_1=p_2=2\) \\
\hline
\multirow{4}{*}{\(|\Gamma_3|\)}
& \(0<\alpha\le\alpha_\star\) & \(\displaystyle\frac{\alpha}{6}\) & \(f_{\rm tri}\), \(\tau_1=\tau_2=0,\ \tau_3=1\) \\
\cline{2-4}
& \(\alpha_\star<\alpha\le\alpha_b\) & \(\displaystyle\frac{\alpha(7\alpha+2)}{18}\sqrt{\frac{2(7\alpha+2)}{29\alpha^2+42\alpha+12}}\) & \(\tau_1=\tau_d^*,\ \tau_2=1\) \\
\cline{2-4}
& \(\alpha_b<\alpha<\alpha_c\) & \(\displaystyle\frac{29\alpha(49\alpha^2-16)^{3/2}}{2232\sqrt{29\alpha^2-12}}\) & \(\tau_1=\sqrt{x_e^*},\ \tau_2=e^{i\theta}\), \(\cos\theta=\frac{7[12-(29\alpha^2+12)x_e^*]}{232\alpha x_e^*}\) \\
\cline{2-4}
& \(\alpha_c\le\alpha\le1\) & \(\displaystyle\frac{29\alpha^3}{72}\) & \(f_{\rm Koebe}\), \(p_1=p_2=p_3=2\) \\
\hline
\multirow{2}{*}{\(|\Gamma_2|-|\Gamma_1|\)}
& \(0<\alpha\le1\), upper & \(\displaystyle\frac{\alpha}{4}\) & \(f_{\rm odd}\), \(p_1=0,\ p_2=2\) \\
\cline{2-4}
& \(0<\alpha\le1\), lower & \(\displaystyle-\frac{\alpha}{2}\sqrt{\frac{2}{3\alpha+2}}\) & \(f_{\rm spiral}\), \(\omega(z)=\frac{z(\sqrt2+\sqrt{3\alpha+2}z)}{\sqrt{3\alpha+2}+\sqrt2 z}\) \\
\hline
\multirow{2}{*}{\(|H_{2,1}(F_{f^{-1}}/2)|\)}
& \(0<\alpha\le\frac25\) & \(\displaystyle\frac{\alpha^2}{16}\) & \(f_{\rm odd}\), \(p_1=0,\ p_2=2\) \\
\cline{2-4}
& \(\frac25<\alpha\le1\) & \(\displaystyle\frac{\alpha^2(15\alpha^2+10\alpha+4)}{4(35\alpha^2+60\alpha+12)}\) & \(f_{\rm Hankel}\), \(\tau_1=\sqrt{y_v},\ \tau_2=1\) \\
\hline
\multirow{3}{*}{\(T_{3,1}(f)\)}
& \(0<\alpha\le1\), upper & \(1\) & \(f_{\rm tri}\), \(p_1=p_2=0\) \\
\cline{2-4}
& \(0<\alpha<\alpha_H\), lower & \(\displaystyle1-2\alpha^2+\frac{15\alpha^4}{16}\) & \(f_{\rm Koebe}\), \(p_1=2,\ \xi=-1\) \\
\cline{2-4}
& \(\alpha_H\le\alpha\le1\), lower & \(\displaystyle1-\frac{\alpha^2}{4}-\frac{\alpha^2(\alpha+6)^2}{4(15\alpha^2+4\alpha-4)}\) & \(p_1=\sqrt{x_v},\ \xi=-1\) \\
\hline
\multirow{5}{*}{\(|a_3-\lambda a_2^2|-\mu|a_2|\)}
& \(D<2+4\mu\), upper & \(\displaystyle\frac{\alpha}{2}\) & \(f_{\rm odd}\), \(p_1=0,\ p_2=2\) \\
\cline{2-4}
& \(D\ge2+4\mu\), upper & \(\displaystyle\frac{\alpha D-4\alpha\mu}{4}\) & \(f_{\rm Koebe}\), \(p_1=p_2=2\) \\
\cline{2-4}
& \(D\le2\mu-2\), lower & \(\displaystyle-\frac{\alpha}{4}(4\mu-D)\) & \(f_{\rm Koebe}\), \(p_1=p_2=2\) \\
\cline{2-4}
& \(D\ge2\mu^2-2\), lower & \(\displaystyle-\alpha\mu\sqrt{\frac{2}{D+2}}\) & \(p_1=2\sqrt{\frac{2}{D+2}},\ \tau_2=-\eta\) \\
\cline{2-4}
& \(2\mu-2<D<2\mu^2-2\), lower & \(\displaystyle-\frac{\alpha}{2}-\frac{\alpha\mu^2}{D+2}\) & \(p_1=\frac{4\mu}{D+2},\ \tau_2=-\eta\) \\
\hline
\end{tabular}
\end{table}

\smallskip
\noindent
In Table~\ref{tab:summary}, the constants are defined by
\[
D=\alpha|3-4\lambda|,
\qquad
\eta=
\begin{cases}
\dfrac{3-4\lambda}{|3-4\lambda|}, & \lambda\neq\dfrac34,\\[8pt]
1, & \lambda=\dfrac34,
\end{cases}
\]
\[
\alpha_H=\frac{-1+\sqrt{241}}{15},
\qquad
y_v=\frac{6(5\alpha-2)}{35\alpha^2+60\alpha+12},
\]
and \(\alpha_\star,\alpha_b,\alpha_c,\tau_d^*,x_e^*,\theta\) are as defined in Theorem~\ref{T1}. The functions \(f_{\rm tri}\), \(f_{\rm odd}\), \(f_{\rm Koebe}\), \(f_{\rm spiral}\), and \(f_{\rm Hankel}\) are obtained from the general integral representation displayed above, with the corresponding Carath\'eodory parameters specified in the table.

\section*{Declarations}

\subsection*{Funding}
The first author acknowledges financial support from the Council of Scientific and Industrial Research (CSIR), New Delhi, India, under Grant No.\ 09/1224(16975)/2023-EMR-I.

\subsection*{Data Availability Statement}
Data sharing is not applicable to this article as no datasets were generated or analyzed during the current study.

\subsection*{Conflict of Interest and Author Contributions}
Both authors contributed equally to this work. The authors declare that they have no conflict of interest.

\end{document}